\documentclass[amstex,12pt,russian,amssymb]{article}

\usepackage{mathtext}
\usepackage[cp1251]{inputenc}
\usepackage[T2A]{fontenc}
\usepackage[russian]{babel}
\usepackage[dvips]{graphicx}
\usepackage{amsmath}
\usepackage{amssymb}
\usepackage{amsxtra}
\usepackage{latexsym}
\usepackage{ifthen}

\begin{document}

\newcounter{lemma}
\newcommand{\lemma}{\par \refstepcounter{lemma}%
{\bf Лема \arabic{lemma}.}}

\newcounter{corollary}
\newcommand{\corollary}{\par \refstepcounter{corollary}%
{\bf Наслідок \arabic{corollary}.}}

\newcounter{remark}
\newcommand{\remark}{\par \refstepcounter{remark}%
{\bf Зауваження \arabic{remark}.}}

\newcounter{theorem}
\newcommand{\theorem}{\par \refstepcounter{theorem}%
{\bf Теорема \arabic{theorem}.}}

\newcounter{proposition}
\newcommand{\proposition}{\par \refstepcounter{proposition}%
{\bf Твердження \arabic{proposition}.}}

\newcounter{example}
\newcommand{\example}{\par \refstepcounter{example}%
{\bf Приклад \arabic{example}.}}

\renewcommand{\refname}{\centerline{\bf Список літератури}}

\renewcommand{\figurename}{Мал.}

\newcommand{\proof}{{\it Доведення.\,\,}}

\noindent УДК 517.5

{\bf Е.А.~Севостьянов} (Житомирский государственный университет
имени Ивана Франко; Институт прикладной математики и механики НАН
Украины, г.~Славянск)

\medskip\medskip
{\bf Є.О.~Севостьянов} (Житомирський державний університет імені
Івана Фран\-ка; Інститут прикладної математики і механіки НАН
України, м.~Слов'янськ)

\medskip\medskip
{\bf E.A.~Sevost'yanov} (Zhytomyr Ivan Franko State University;
Institute of Applied Ma\-the\-ma\-tics and Mechanics of NAS of
Ukraine, Slov'yans'k)

\medskip
{\bf Граничное продолжение отображений с обратным неравенством
Полецкого по простым концам}

{\bf Межове продовження відображень з оберненою нерівністю
Полецького по простих кінцях}

{\bf Boundary extension of mappings with the inverse Poletsky
inequality by prime ends}

\medskip\medskip
Для отображений с ветвлением, удовлетворяющих обратному неравенству
Полецкого, получены результаты об их непрерывном граничном
продолжении в терминах простых концов. При определённых условиях
показано, что указанные классы отображений являются также
равностепенно непрерывными в замыкании заданной области.

\medskip\medskip
Для відображень із розгалуженням, які задовольняють обернену
нерівність Полецького, отримано результати про їх неперервне межове
продовження в термінах простих кінців. За певних умов вказані класи
відображень є також одностайно неперервними в замиканні заданої
області.

\medskip\medskip
For mapping with branching points that satisfy the inverse
inequality of Poletsky, we obtained the results of their continuous
boundary extension in terms of prime ends. Under certain conditions,
the specified classes od mappings are also equicontinuous in the
closure of a given domain.

\newpage
{\bf 1. Вступ.} В наших спільних роботах~\cite{SalSev} і~\cite{SSI}
отримано неперервне продовження на межу і одностайну неперервність
гомеоморфізмів, обернені до яких задовольняють певну оцінку
спотворення модуля сімей кривих. Мова йшла про області з поганими
межами, відносно яких відображення не має звичайного неперервного
продовження, але має його в узагальненому сенсі, точніше -- сенсі
так званих простих кінців. В даній замітці ми встановимо аналогічний
результат для відображень з розгалуженням, які, як правило,
припускаються відкритими, дискретними і замкненими (зберігаючими
межу області). Підкреслимо, що ситуація гомеоморфізмів, докладно
розібрана в роботах~\cite{SalSev} і~\cite{SSI}, випливає з наших
основних теорем як наслідок; в той самий час, основна умова щодо
спотворення модуля дещо більш загальна у порівнянні з~\cite{SalSev}
і \cite{SSI}. З приводу деяких відомих результатів стосовно
неперервного продовження квазіконформних відображень і їх
узагальнень по простих кінцях вкажемо, напр., на праці
\cite{GRY}--\cite{Na}.

\medskip
Наведемо деякі означення і позначення. Нехай $y_0\in {\Bbb R}^n,$
$0<r_1<r_2<\infty$ і
\begin{equation}\label{eq1**}
A(y_0, r_1,r_2)=\left\{ y\,\in\,{\Bbb R}^n:
r_1<|y-y_0|<r_2\right\}\,.\end{equation}
Якщо $f:D\rightarrow {\Bbb R}^n$ -- задане відображення, $y_0\in
f(D)$ і $0<r_1<r_2<d_0=\sup\limits_{y\in f(D)}|y-y_0|,$ то через
$\Gamma_f(y_0, r_1, r_2)$ ми позначимо сім'ю всіх кривих $\gamma$ в
області $D$ таких, що $f(\gamma)\in \Gamma(S(y_0, r_1), S(y_0, r_2),
A(y_0,r_1,r_2)).$ Нехай $Q:{\Bbb R}^n\rightarrow [0, \infty]$ --
вимірна за Лебегом функція.  Будемо говорити, що {\it $f$
задовольняє обернену нерівність Полецького} в точці $y_0\in f(D),$
якщо співвідношення
\begin{equation}\label{eq2*A}
M(\Gamma_f(y_0, r_1, r_2))\leqslant \int\limits_{f(D)\cap A(y_0,
r_1,r_2)} Q(y)\cdot \eta^n (|y-y_0|)\, dm(y)
\end{equation}
виконується для довільної вимірної за Лебегом функції $\eta:
(r_1,r_2)\rightarrow [0,\infty ]$ такій, що
\begin{equation}\label{eqA2}
\int\limits_{r_1}^{r_2}\eta(r)\, dr\geqslant 1\,.
\end{equation}
Зауважимо, що нерівності~(\ref{eq2*A}) добре відомі в теорії
квазірегулярних відображень і виконуються для них при $Q=N(f,
D)\cdot K, $ де $N(f, D)$ -- максимальна кратність відображення в
$D,$ а $K\geqslant 1$ -- деяка стала, яка може бути обчислена як
$K={\rm ess \sup}\, K_O(x, f),$ $K_O(x, f)=\Vert
f^{\,\prime}(x)\Vert^n/|J(x, f)|$ при $J(x, f)\ne 0;$ $K_O(x, f)=1$
при $f^{\,\prime}(x)=0,$ і $K_O(x, f)=\infty$ при
$f^{\,\prime}(x)\ne 0,$ але $J(x, f)=0$ (див., напр.,
\cite[теорема~3.2]{MRV$_1$} або \cite[теорема~6.7.II]{Ri}).
Відображення $f:D\rightarrow {\Bbb R}^n$ називається {\it
дискретним}, якщо прообраз $\{f^{-1}\left(y\right)\}$ кожної точки
$y\,\in\,{\Bbb R}^n$ складається з ізольованих точок, і {\it
відкритим}, якщо образ будь-якої відкритої множини $U\subset D$ є
відкритою множиною в ${\Bbb R}^n.$ Відображення $f$ області $D$ на
$D^{\,\prime}$ називається {\it замкненим}, якщо $f(E)$ є замкненим
в $D^{\,\prime}$ для будь-якої замкненої множини $E\subset D$ (див.,
напр., \cite[розд.~3]{Vu$_1$}).

\medskip
Нехай $\omega$ --  відкрита множина в ${\Bbb R}^k$,
$k=1,\ldots,n-1$. Неперервне відображення
$\sigma\colon\omega\rightarrow{\Bbb R}^n$ називається {\it
$k$-вимірною поверхнею} в ${\Bbb R}^n$. {\it Поверхнею} будемо
називати довільну $(n-1)$-вимірну поверхню $\sigma$ в ${\Bbb R}^n.$
Поверхня $\sigma$ називається {\it жордановою поверхнею}, якщо
$\sigma(x)\ne\sigma(y)$ при $x\ne y$. Далі ми іноді будемо
використовувати $\sigma$ для позначення всього образу
$\sigma(\omega)\subset {\Bbb R}^n$ при відображенні $\sigma$,
$\overline{\sigma}$ замість $\overline{\sigma(\omega)}$ в ${\Bbb
R}^n$ і $\partial\sigma$ замість
$\overline{\sigma(\omega)}\setminus\sigma(\omega)$. Жорданова
поверхня $\sigma\colon\omega\rightarrow D$ в області $D$ називається
{\it розрізом} області $D$, якщо $\sigma$ розділяє $D$, тобто
$D\setminus \sigma$ має більше однієї компоненти,
$\partial\sigma\cap D=\varnothing$ і $\partial\sigma\cap\partial
D\ne\varnothing$.

Послідовність $\sigma_1,\sigma_2,\ldots,\sigma_m,\ldots$ розрізів
області $D$ називається {\it ланцюгом}, якщо:

(i) множина $\sigma_{m+1}$ міститься в точності в одній компоненті
$d_m$ множини $D\setminus \sigma_m$, при цьому, $\sigma_{m-1}\subset
D\setminus (\sigma_m\cup d_m)$; (ii)
$\bigcap\limits_{m=1}^{\infty}\,d_m=\varnothing$. З означення
ланцюгу розрізів випливає, що $d_1\supset d_2\supset
d_3\supset\ldots\supset d_{m-1}\supset d_m\supset
d_{m+1}\supset\ldots\,.$
Два ланцюги розрізів $\{\sigma_m\}$ і $\{\sigma_k^{\,\prime}\}$
називаються {\it еквівалентними}, якщо для кожного $m=1,2,\ldots$
область $d_m$ містить всі області $d_k^{\,\prime}$ за виключенням
скінченної кількості, і для кожного $k=1,2,\ldots$ область
$d_k^{\,\prime}$ також містить всі області $d_m$ за виключенням
скінченної кількості.

{\it Кінець} області $D$ --- це клас еквівалентних ланцюгів розрізів
області $D$. Нехай $K$ --- кінець області $D$ в ${\Bbb R}^n$, тоді
множина $I(K)=\bigcap\limits_{m=1}\limits^{\infty}\overline{d_m}$
називається {\it тілом кінця} $K$. Скрізь далі, як зазвичай,
$\Gamma(E, F, D)$ позначає сім'ю всіх таких кривих $\gamma\colon[a,
b]\rightarrow D$, що $\gamma(a)\in E$ і $\gamma(b)\in F,$ крім того,
$M(\Gamma)$ позначає модуль сім'ї кривих $\Gamma$ в ${\Bbb R}^n,$ а
запис $\rho\in {\rm adm}\,\Gamma$ означає, що функція $\rho$
борелева, невід'ємна і має довжину, не меншу ніж одиницю, в метриці
$\rho$ (див.~\cite{Na}, \cite{Va}). Слідуючи~\cite{Na}, будемо
говорити, що кінець $K$ є {\it простим кінцем}, якщо $K$ містить
ланцюг розрізів $\{\sigma_m\}$, такий, що $M(\Gamma(\sigma_m,
\sigma_{m+1}, D))<\infty$ при всіх $m\in {\Bbb N}$ і $
\lim\limits_{m\rightarrow\infty}M(\Gamma(C, \sigma_m, D))=0 $
для деякого континууму $C$ в $D$ (див. малюнок~\ref{fig1}).
\begin{figure}
\centering\includegraphics[scale=0.5]{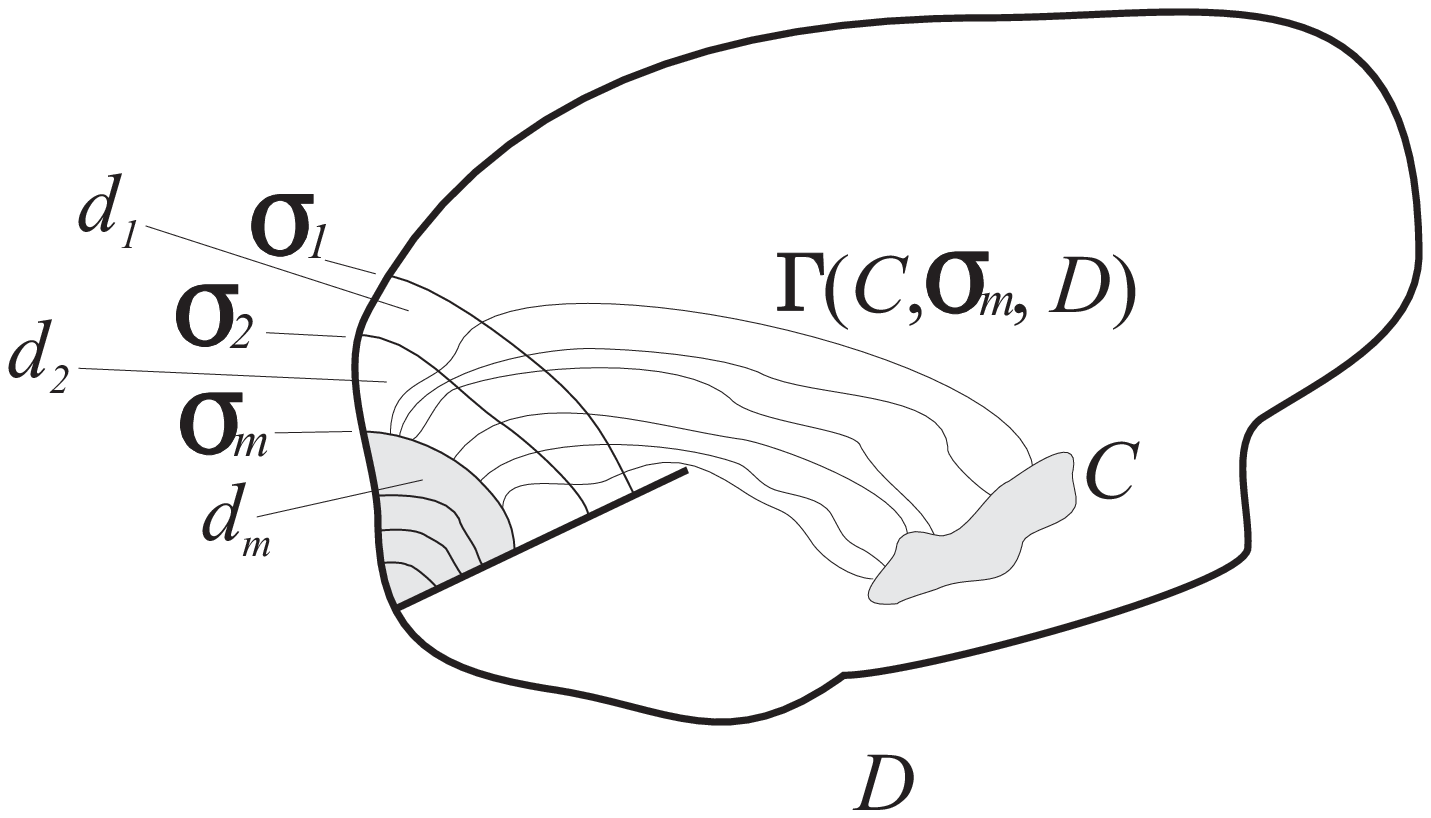}
\caption{Простий кінець в області}\label{fig1}
 \end{figure}
Далі використовуються наступні позначення: множина простих кінців,
що відповідають області $D,$ позначається символом $E_D,$ а
поповнення області $D$ її простими кінцями позначається
$\overline{D}_P.$ Будемо говорити, що межа області $D$ в ${\Bbb
R}^n$ є {\it локально квазіконформною}, якщо кожна точка
$x_0\in\partial D$ має окіл $U$ в ${\Bbb R}^n$, який може бути
відображений квазіконформним відображенням $\varphi$ на одиничну
кулю ${\Bbb B}^n\subset{\Bbb R}^n$ так, що $\varphi(\partial D\cap
U)$ є перетином ${\Bbb B}^n$ з координатною гіперплощиною.
Розглянемо також наступне означення (див.~\cite{KR$_1$},
\cite{KR$_2$}). Для множин $E\subset {\Bbb R}^n$ і $A, B\subset
{\Bbb R}^n$ покладемо
$$d(E):=\sup\limits_{x, y\in E}|x-y|\,,\quad d(A, B):=\inf\limits_{x\in A, y\in B}|x-y|\,.$$
Будемо називати ланцюг розрізів $\{\sigma_m\}$ {\it регулярним},
якщо $\overline{\sigma_m}\cap\overline{\sigma_{m+1}}=\varnothing$
при кожному $m\in {\Bbb N}$ і, крім того, $d(\sigma_{m})\rightarrow
0$ при $m\rightarrow\infty.$ Якщо кінець $K$ містить принаймні один
регулярний ланцюг, то $K$ будемо називати {\it регулярним}.
Говоримо, що обмежена область $D$ в ${\Bbb R}^n$ {\it регулярна},
якщо $D$ може бути квазіконформно відображена на область з локально
квазіконформною межею, замикання якої є компактом в ${\Bbb R}^n,$
крім того, кожен простий кінець $P\subset E_D$ є регулярним.
Зауважимо, що у просторі ${\Bbb R}^n$ кожний простий кінець
регулярної області містить ланцюг розрізів з властивістю
$d(\sigma_{m})\rightarrow 0$ при $m\rightarrow\infty,$ і навпаки,
якщо у кінця є вказана властивість, то він -- простий
(див.~\cite[теорема~5.1]{Na}). Крім того, замикання $\overline{D}_P$
регулярної області $D$ є {\it метризовним}, при цьому, якщо
$g:D_0\rightarrow D$ -- квазіконформне відображення області $D_0$ з
локально квазіконформною межею на область $D,$ то для $x, y\in
\overline{D}_P$ покладемо:
\begin{equation}\label{eq1A}
\rho(x, y):=|g^{\,-1}(x)-g^{\,-1}(y)|\,,
\end{equation}
де для $x\in E_D$ елемент $g^{\,-1}(x)$ розуміється як деяка (єдина)
точка межі $D_0,$ коректно визначена з огляду
на~\cite[теорема~4.1]{Na}. Зокрема, будемо говорити, що
послідовність $x_m\in D,$ $m=1,2,\ldots,$ {\it збігається} до
простого кінця $P\in E_D$ при $m\rightarrow\infty,$ якщо для
будь-якого натурального $k\in {\Bbb N}$ всі елементи послідовності
$x_m,$ крім скінченної кількості, належать області $d_k$ (де $d_k,$
$k=1,2,\ldots$ -- послідовність вкладених областей з означення
простого кінця $P$). Якщо, наприклад, $f$ -- гомеоморфізм області
$D$ на $D^{\,\prime},$ то не важко переконатися, що між кінцями
областей $D$ і $D^{\,\prime}=f(D)$ є взаємно однозначна
відповідність (див. малюнок~\ref{fig2}).
\begin{figure}[h]
\centering\includegraphics[width=300pt]{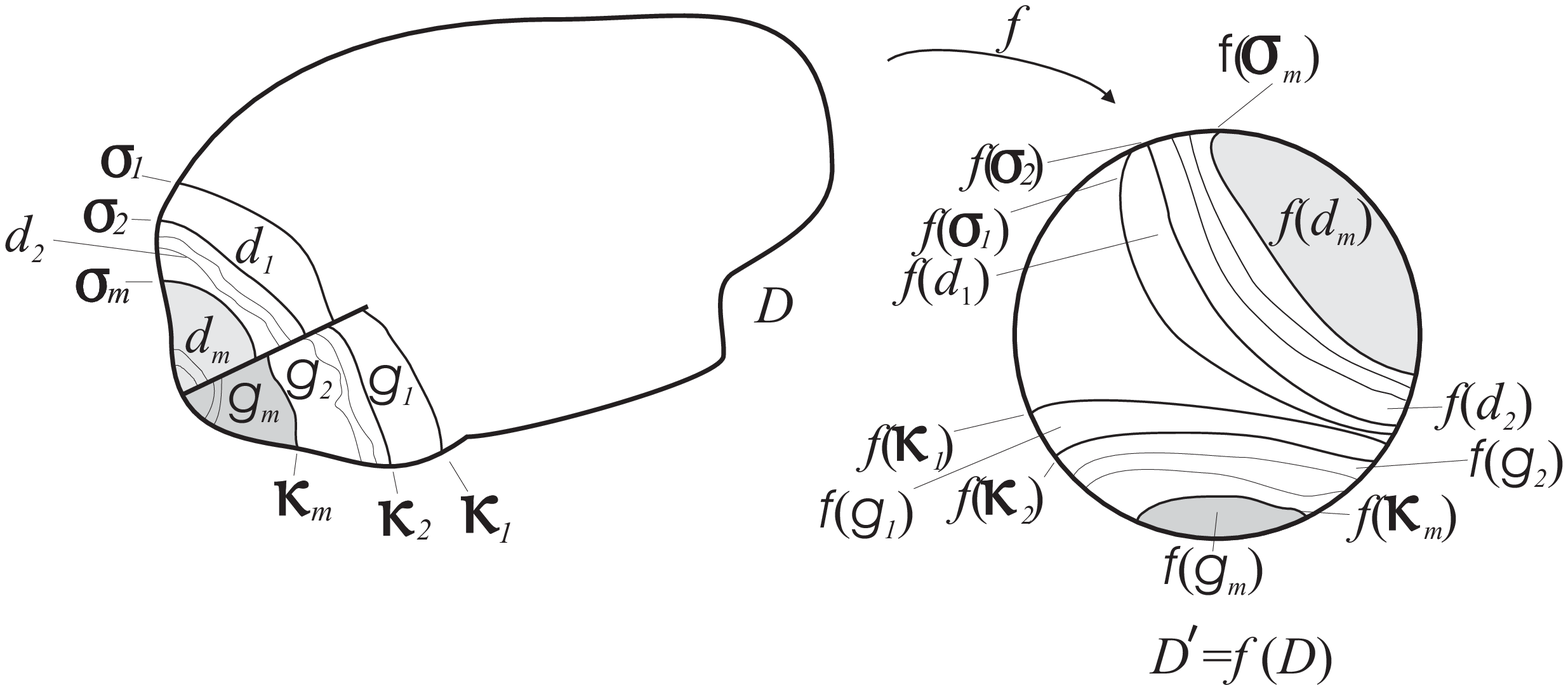}
\caption{Відповідність простих кінців при відображенні}\label{fig2}
\end{figure}
Є справедливим такий результат.

\medskip
\begin{theorem}\label{th3}
{\sl Нехай $D\subset {\Bbb R}^n,$ $n\geqslant 2,$ -- область, яка
має слабо плоску межу, а область $D^{\,\prime}\subset {\Bbb R}^n$ є
регулярною. Припустимо, $f$ -- відкрите дискретне і замкнене
відображення області $D$ на $D^{\,\prime},$ що задовольняє
співвідношення~(\ref{eq2*A}) в кожній точці $y_0\in D^{\,\prime},$
де $Q\in L^1(D^{\,\prime}).$ Тоді відображення $f$ має неперервне
продовження до відображення
$\overline{f}:\overline{D}\rightarrow\overline{D^{\,\prime}}_P,$
причому, $\overline{f}(\overline{D})=\overline{D^{\,\prime}}_P.$}
\end{theorem}

\medskip
Справедливим також є результат про одностайну неперервність сімей
відображень виду~(\ref{eq2*A}) в замиканні даної області. З метою
формулювання відповідного результату, розглянемо такі означення.

\medskip
Межа області $D$ називається {\it слабо плоскою} в точці $x_0\in
\partial D,$ якщо для кожного $P>0$ і для будь-якого околу $U$
точки $x_0$ знайдеться окіл $V\subset U$ цієї ж самої точки такий,
що $M(\Gamma(E, F, D))>P$ для будь-яких континуумів $E, F\subset D,$
які перетинають $\partial U$ і $\partial V.$ Межа області $D$
називається слабо плоскою, якщо відповідна властивість виконується в
будь-якій точці межі $D.$ Нехай $h$ -- хордальна відстань в
$\overline{{\Bbb R}^n}$ (див., напр., означення~12.1 в~\cite{Va}). У
подальшому, для множин $A, B\subset \overline{{\Bbb R}^n}$ покладемо
$$h(A, B)=\inf\limits_{x\in A, y\in B}h(x, y)\,,
\quad h(A)=\sup\limits_{x, y\in A}h(x ,y)\,,$$
де $h$ -- хордальная відстань. Для числа $\delta>0,$ областей $D,
D^{\,\prime}\subset {\Bbb R}^n,$ $n\geqslant 2,$ континуума
$A\subset D^{\,\prime}$ і довільної вимірної за Лебегом функції
$Q:D^{\,\prime}\rightarrow [0, \infty]$ позначимо через ${\frak
S}_{\delta, A, Q }(D, D^{\,\prime})$ сім'ю всіх відкритих дискретних
і замкнених відображень $f$ області $D$ на $D^{\,\prime},$ що
задовольняють умову~(\ref{eq2*A}) для кожного $y_0\in D^{\,\prime}$
і таких, що $h(f^{\,-1}(A),
\partial D)\geqslant~\delta.$  Виконується
наступне твердження.

\medskip
\begin{theorem}\label{th2}
{\sl Припустимо, що область $D$ має слабо плоску межу. Якщо $Q\in
L^1(D^{\,\prime}),$ і область $D^{\,\prime}$ є регулярною, то
будь-яке $f\in{\frak S}_{\delta, A, Q }(D, D^{\,\prime})$ неперервно
продовжується до відображення $\overline{f}:\overline{D}\rightarrow
\overline{D^{\,\prime}}_P,$ причому,
$\overline{f}(\overline{D})=\overline{D^{\,\prime}}_P$ і сім'я
${\frak S}_{\delta, A, Q }(\overline{D}, \overline{D^{\,\prime}}),$
яка складається з усіх продовжених відображень
$\overline{f}:\overline{D}\rightarrow \overline{D^{\,\prime}}_P,$
одностайно неперервна в $\overline{D}.$ }
\end{theorem}

\medskip
Зауважимо, що для випадку гарних меж твердження теорем~\ref{th3} і
\ref{th2} доведено раніше в роботі~\cite{SSD} (див. теореми~1.1 і
1.2).

\medskip
{\bf 2. Доведення теореми~\ref{th3}.} Зафіксуємо довільним чином
точку $x_0\in\partial D.$ Необхідно показати можливість неперервного
продовження відображення $f$ в точку $x_0.$ Використовуючи при
необхідності мебіусове перетворення $\varphi:\infty\mapsto 0$ і
враховуючи інваріантість модуля $M$ в лівій частині
співвідношення~(\ref{eq2*A}) (див.~\cite[теорема~8.1]{Va}), ми
можемо вважати, що $x_0\ne\infty.$

\medskip
Припустимо, що висновок про неперервне продовження відображення $f$
в точку $x_0$ не є правильним. Тоді будь-який простий кінець $P_0\in
E_{D^{\,\prime}}$ не є границею $f$ в точці $x_0,$ тобто, знайдеться
послідовність $x_k\rightarrow x_0$ при $k\rightarrow\infty$ і число
$\varepsilon_0>0$ такі, що $\rho(f(x_k), P_0)\geqslant
\varepsilon_0$ при всіх $k\in {\Bbb N},$ де $\rho$ -- одна з метрик
в~(\ref{eq1A}). Оскільки за умовою область $D^{\,\prime}$ є
регулярною, її можна відобразити на обмежену область $D_*$ за
допомогою деякого квазіконформного відображення $h:
D^{\,\prime}\rightarrow D_*.$ Оскільки між точками межі областей з
локально квазіконформними межами і їх простими кінцями є взаємно
однозначна відповідність (див.~\cite[теорема~4.1]{Na}), і за умовою
$\overline{D}_*$ є компактом в ${\Bbb R}^n,$ метричний простір
$(\overline{D^{\,\prime}}_P, \rho)$ є компактним. Отже, можна
вважати, що $f(x_k)$ збігається до якогось елементу $P_1\ne P_0,$
$P_1\in\overline{D^{\,\prime}}_P$ при $k\rightarrow\infty.$ Оскільки
за припущенням відображення $f$ не має границі в точці $x_0,$ існує
принаймні ще одна послідовність $y_k\rightarrow x_0$ при
$k\rightarrow\infty,$ така що $\rho(f(y_k), P_1)\geqslant
\varepsilon_1$ при всіх $k\in {\Bbb N}$ і деякому $\varepsilon_1>0.$
Знову таки, оскільки метричний простір $(\overline{D^{\,\prime}}_P,
\rho)$ є компактним, ми можемо вважати, що $f(y_k)\rightarrow P_2$
при $k\rightarrow \infty,$ $P_1\ne P_2,$ $P_2\in
\overline{D^{\,\prime}}_P.$ Оскільки відображення $f$ замкнене, воно
зберігає межу області, див.~\cite[теорема~3.3]{Vu$_1$}. Отже, $P_1,
P_2\in E_{D^{\,\prime}}.$

\medskip
Нехай $\sigma_m$ і $\sigma^{\,\prime}_m,$ $m=0,1,2,\ldots, $  --
послідовності розрізів, які відповідають простим кінцям $P_1$ і
$P_2,$ відповідно. Нехай також розрізи $\sigma_m,$ $m=0,1,2,\ldots,
$ лежать на сферах $S(z_0, r_m)$ з центром в деякій точці $z_0\in
\partial D^{\,\prime},$ де $r_m\rightarrow 0$ при $m\rightarrow\infty$
(така послідовність $\sigma_m$ існує за~\cite[лема~3.1]{IS}, див.
також~\cite[лема~1]{KR$_2$}). Нехай $d_m$ і $g_m,$ $m=0,1,2,\ldots,
$ -- відповідні послідовності областей в $D^{\,\prime},$ що
відповідають розрізам $\sigma_m$ і $\sigma^{\,\prime}_m,$
відповідно. Оскільки простір $(\overline{D^{\,\prime}}_P, \rho)$ є
метричним, можна вважати, що всі $d_m$ і $g_m$ не перетинаються між
собою для кожного $m=0,1,2,\ldots ,$ зокрема,
\begin{equation}\label{eq4}
d_0\cap g_0=\varnothing\,.
\end{equation}
Оскільки $f(x_k)$ збігається до $P_1$ при $k\rightarrow\infty,$ для
кожного $m\in {\Bbb N}$ існує $k=k(m):$ $f(x_k)\in d_m$ при
$k\geqslant k=k(m).$ Шляхом перенумерації послідовності $x_k$ в разі
необхідності, ми можемо добитися того, щоб $f(x_k)\in d_k$ при
кожному натуральному $k.$ Аналогічно, можна вважати, що $f(y_k)\in
g_k$ при всіх $k\in {\Bbb N}.$ Зафіксуємо точки $f(x_1)$ і $f(y_1).$
Оскільки за означенням простого кінця
$\bigcap\limits_{k=1}^{\infty}d_k=\bigcap\limits_{l=1}^{\infty}g_l=\varnothing,$
існують номери $k_1$ і $k_2\in {\Bbb N}$ такі, що $f(x_1)\not\in
d_{k_1}$ і $f(y_1)\not \in g_{k_2}.$ Оскільки за означенням
послідовності областей $d_k\subset d_{k_0}$ при всіх $k\geqslant
k_1$ і $g_k\subset g_{k_2}$ при $k\geqslant k_2,$ будемо мати
\begin{equation}\label{eq3}
f(x_1)\not\in d_k\,,\quad f(y_1)\not\in g_k\,, \quad
k\geqslant\max\{k_1, k_2\}\,.
\end{equation}

\medskip
Нехай $\gamma_k$ -- крива, що з'єднує $f(x_1)$ і $f(x_k)$ в області
$d_1,$ а $\gamma^{\,\prime}_k$ -- крива, що з'єднує $f(y_1)$ і
$f(y_k)$ в області $g_1.$ Нехай також $\alpha_k$ і $\beta_k$ --
повні $f$-підняття кривих $\gamma_k$ і $\gamma^{\,\prime}_k$ в
області $D$ з початками в точках $x_k$ і $y_k,$ відповідно (такі
підняття існують за~\cite[лема~3.7]{Vu$_1$}),
див.~малюнок~\ref{fig1A}).
\begin{figure}[h]
\centerline{\includegraphics[scale=0.5]{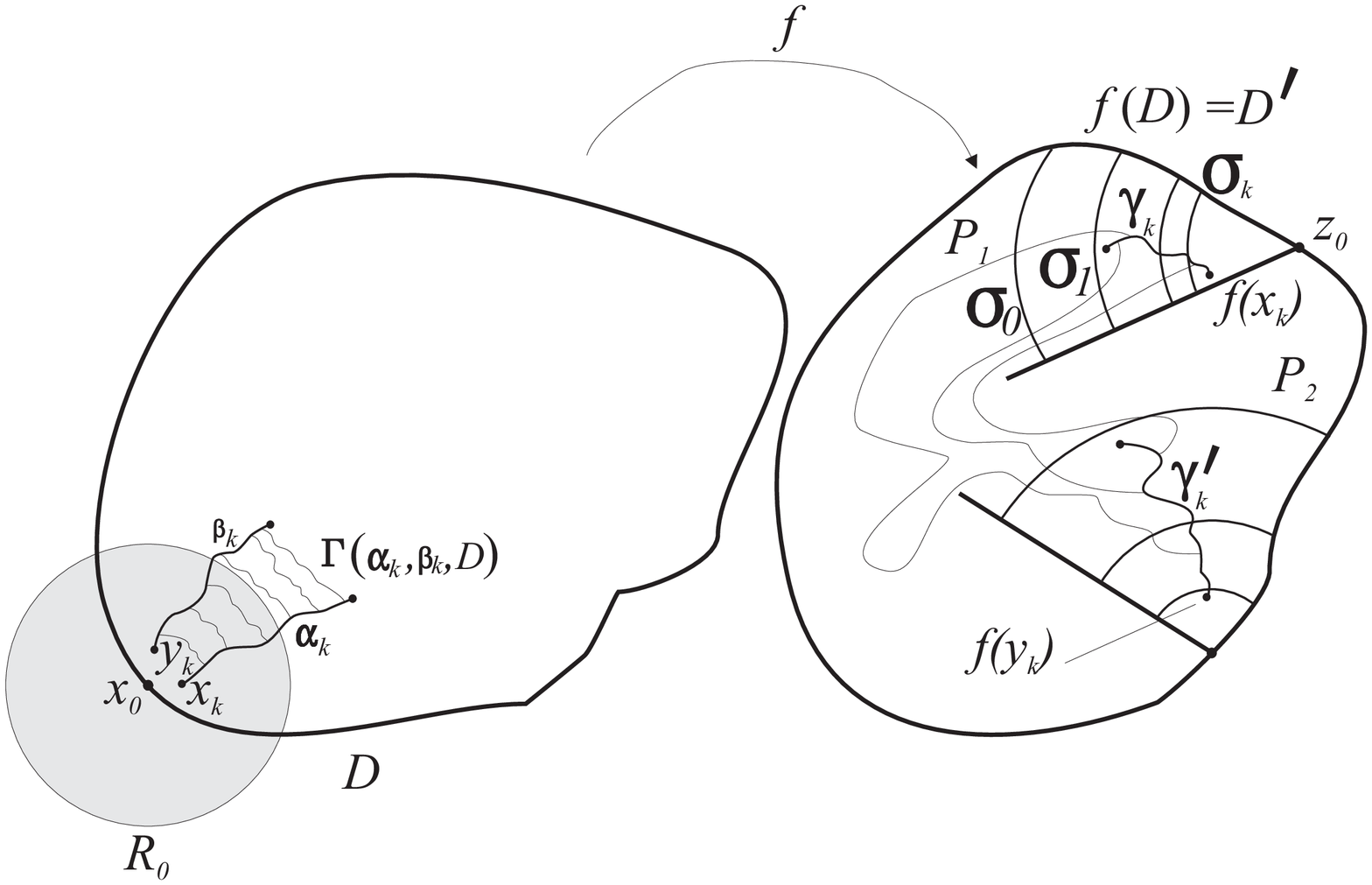}} \caption{До
доведення теореми~\ref{th3}}\label{fig1A}
\end{figure}
Зауважимо, що у точок $f(x_1)$ і $f(y_1)$ в області $D$ може бути не
більше скінченного числа прообразів при відображенні $f,$
див.~\cite[лема~3.2]{Vu$_1$}. Тоді знайдеться $R_0>0$ таке, що
$\alpha_k(1), \beta_k(1)\in D\setminus B(x_0, R_0)$ при всіх
$k=1,2,\ldots .$ Оскільки межа області $D$ є слабо плоскою, для
кожного $P>0$ знайдеться $k=k_P\geqslant 1$ таке, що
\begin{equation}\label{eq7}
M(\Gamma(|\alpha_k|, |\beta_k|, D))>P\qquad\forall\,\,k\geqslant
k_P\,.
\end{equation}
Покажемо, що умова~(\ref{eq7}) суперечить визначенню відображення
$f$ в~(\ref{eq2*A}). Справді, нехай $\gamma\in \Gamma(|\alpha_k|,
|\beta_k|, D),$ тоді $\gamma:[0, 1]\rightarrow D,$ $\gamma(0)\in
|\alpha_k|$ і $\gamma(0)\in |\beta_k|.$ Зокрема, $f(\gamma(0))\in
|\gamma_k|$ і $f(\gamma(1))\in |\gamma^{\,\prime}_k|.$ В такому
випадку, зі співвідношень~(\ref{eq4}) і~(\ref{eq7}) випливає, що
$|f(\gamma)|\cap d_1\ne\varnothing \ne |f(\gamma)|\cap(D\setminus
d_1)$ при $k\geqslant\max\{k_1, k_2\}.$ З огляду
на~\cite[теорема~1.I.5.46]{Ku} $|f(\gamma)|\cap \partial
d_1\ne\varnothing,$ тобто, $|f(\gamma)|\cap S(z_0,
r_1)\ne\varnothing,$ бо $\partial d_1\cap D\subset \sigma_1\subset
S(z_0, r_1)$ за визначенням розрізу $\sigma_1.$ Нехай $t_1\in (0,1)$
таке, що $f(\gamma(t_1))\in S(z_0, r_1)$ і
$f(\gamma)|_1:=f(\gamma)|_{[t_1, 1]}.$ Без обмеження загальності
можна вважати, що $f(\gamma)|_1\subset {\Bbb R}^n\setminus B(z_0,
r_1).$ Міркуючи так само для кривої $f(\gamma)|_1,$ можна знайти
точку $t_2\in (t_1,1)$ таку, що $f(\gamma(t_2))\in S(z_0, r_0).$
Покладемо $f(\gamma)|_2:=f(\gamma)|_{[t_1, t_2]}.$ Тоді крива
$f(\gamma)|_2$ є підкривою кривої $f(\gamma)$ і, крім того,
$f(\gamma)|_2\in \Gamma(S(z_0, r_1), S(z_0, r_0), D^{\,\prime}).$
Без обмеження загальності можна вважати, що $f(\gamma)|_2\subset
B(z_0, r_0).$ Тим самим
$$\Gamma(|\alpha_k|,
|\beta_k|, D)>\Gamma_f(z_0, r_1, r_0)\,.$$
З останнього співвідношення по міноруванню модуля (див., напр.,
\cite[теорема~1(c)]{Fu})
\begin{equation}\label{eq5}
M(\Gamma(|\alpha_k|, |\beta_k|, D))\leqslant M(\Gamma_f(z_0, r_1,
r_0))\,.
\end{equation}
Покладемо $\eta(t)= \left\{
\begin{array}{rr}
\frac{1}{r_0-r_1}, & t\in [r_1, r_0],\\
0,  &  t\not\in [r_1, r_0]
\end{array}
\right. .$
Зауважимо, що $\eta$ задовольняє співвідношення~(\ref{eqA2}) при
$r_1:=r_1$ і $r_2:=r_0.$ Тоді з~(\ref{eq2*A}) і~(\ref{eq5}) ми
отримаємо, що
\begin{equation}\label{eq11}
M(\Gamma(|\alpha_k|, |\beta_k|, D)) \leqslant
\frac{1}{(r_0-r_1)^n}\int\limits_{D^{\,\prime}}
Q(y)\,dm(y):=c<\infty\quad\forall k\geqslant \max\{k_1, k_2\}\,,
\end{equation}
оскільки~$Q\in L^1(D).$ Співвідношення~(\ref{eq11}) суперечить
умові~(\ref{eq7}). Отримана суперечність спростовує припущення про
відсутність границі у відображення $f$ в точці $x_0.$

Залишилось перевірити рівність
$\overline{f}(\overline{D})=\overline{D^{\,\prime}}_P.$ Очевидно, що
$\overline{f}(\overline{D})\subset\overline{D^{\,\prime}}_P.$
Покажемо, що $\overline{D^{\,\prime}}_P\subset
\overline{f}(\overline{D}).$ Справді, нехай $y_0\in
\overline{D^{\,\prime}}_P,$ тоді або $y_0\in D^{\,\prime},$ або
$y_0\in \partial E_{D^{\,\prime}}.$ Якщо $y_0\in D^{\,\prime},$ то
$y_0=f(x_0)$ і $y_0\in \overline{f}(\overline{D}),$ оскільки за
умовою $f$ -- відображення області $D$ на $D^{\,\prime}.$ Нарешті,
нехай $y_0\in E_{D^{\,\prime}},$ тоді через регулярність
області~$D^{\,\prime}$ знайдеться послідовність $y_k\in
D^{\,\prime}$ така, що $\rho(y_k, y_0)\rightarrow 0$ при
$k\rightarrow\infty,$ $y_k=f(x_k)$ і $x_k\in D,$ де $\rho$ -- одна з
можливих метрик в $\overline{D^{\,\prime}}_P.$ Через компактність
простору $\overline{{\Bbb R}^n}$ ми можемо вважати, що
$x_k\rightarrow x_0,$ де $x_0\in\overline{D}.$ Помітимо, що $x_0\in
\partial D,$ оскільки відображення $f$ є відкритим. Тоді
$f(x_0)=y_0\in \overline{f}(\partial D)\subset
\overline{f}(\overline{D}).$ Теорема повністю доведена.

\medskip
{\bf 3. Допоміжні леми.} Наступну лему доведено
в~\cite[лема~2.1]{SSI$_1$}, див. також~\cite[лема~2.1]{SSI}.

\begin{lemma}\label{lem1}{\sl\,
Нехай $D^{\,\prime}\subset {\Bbb R}^n,$ $n\geqslant 2,$ -- регулярна
область, і нехай $x_m\rightarrow P_1,$ $y_m\rightarrow P_2$ при
$m\rightarrow\infty,$ $P_1, P_2\in \overline{D^{\,\prime}}_P,$
$P_1\ne P_2.$ Припустимо, що $d_m, g_m,$ $m=1,2,\ldots,$ -- дві
послідовності спадних областей, які відповідають $P_1$ і $P_2,$
$d_1\cap g_1=\varnothing,$ і $x_0, y_0\in D^{\,\prime}\setminus
(d_1\cup g_1).$ Тоді існують як завгодно великі номери $k_0\in {\Bbb
N},$ $M_0=M_0(k_0)\in {\Bbb N}$ і $0<t_1=t_1(k_0), t_2=t_2(k_0)<1$
для яких виконано наступну умову: для всякого $m\geqslant M_0$
знайдуться непересічні криві
$$\gamma_{1,m}(t)=\quad\left\{
\begin{array}{rr}
\widetilde{\alpha}(t), & t\in [0, t_1],\\
\widetilde{\alpha_m}(t), & t\in [t_1, 1]\end{array}
\right.\,,\quad\gamma_{2,m}(t)=\quad\left\{
\begin{array}{rr}
\widetilde{\beta}(t), & t\in [0, t_2],\\
\widetilde{\beta_m}(t), & t\in [t_2, 1]\end{array}\,, \right.$$
такі, що

1) $\gamma_{1, m}(0)=x_0,$ $\gamma_{1, m}(1)=x_m,$ $\gamma_{2,
m}(0)=y_0$ і $\gamma_{2, m}(1)=y_m;$

2) $|\gamma_{1, m}|\cap \overline{g_{k_0}}=\varnothing=|\gamma_{2,
m}|\cap \overline{d_{k_0}};$

3) $\widetilde{\alpha_m}(t)\in d_{k_0+1}$ при $t\in [t_1, 1]$ і
$\widetilde{\beta_m}(t)\in g_{k_0+1}$ при $t\in [t_2, 1]$ (див.
малюнок~\ref{fig3}).
\begin{figure}[h]
\centering\includegraphics[width=180pt]{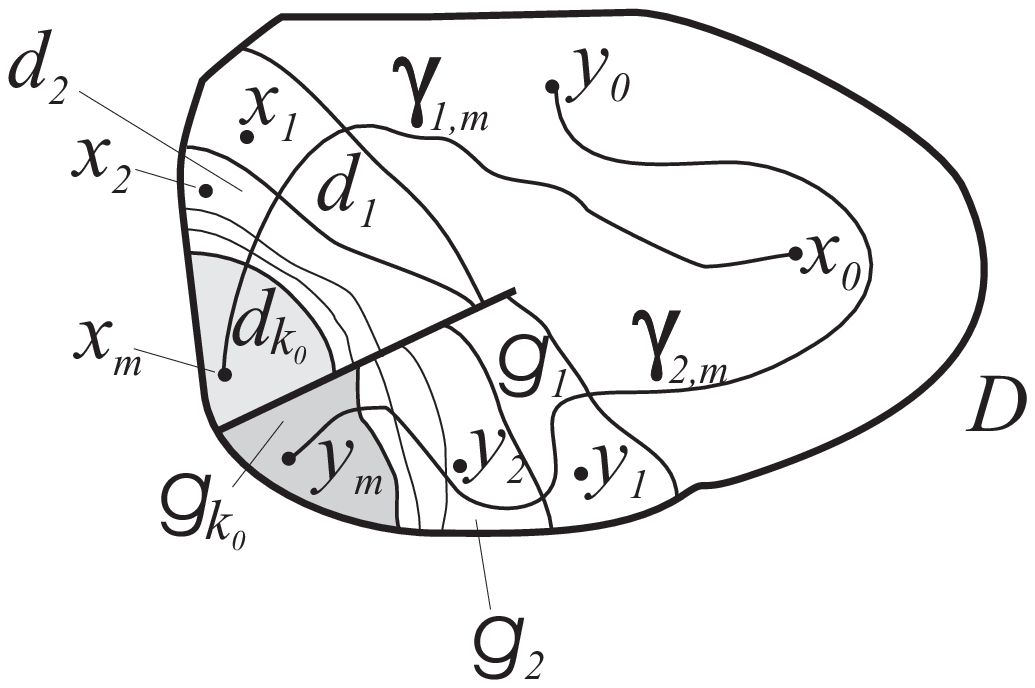}
\caption{Твердження леми~\ref{lem1}}\label{fig3}
\end{figure}
}
\end{lemma}

\medskip
Наступне твердження доведено в~\cite[лема~2.2]{SSI$_1$} для випадку
гомеоморфізмів.

\medskip
\begin{lemma}\label{lem4}
{\sl Нехай $D$ і $D^{\,\prime}$ -- області в ${\Bbb R}^n,$
$n\geqslant 2,$  область $D^{\,\prime}$ є регулярною, і нехай $f$ --
відкрите, дискретне і замкнене відображення області $D$ на
$D^{\,\prime},$ яке задовольняє умову~(\ref{eq2*A}) в кожній точці
$y_0\in\overline{D^{\,\prime}}$ і з деякою функцією $Q\in
L^1(D^{\,\prime}).$ Нехай також $d_m$ -- послідовність спадних
областей, які відповідають ланцюгу розрізів $\sigma_m,$
$m=1,2,\ldots, $ що лежать на сферах $S(\overline{x_0}, r_m)$ і
таких, що $\overline{x_0}\in
\partial D^{\,\prime},$ причому $r_m\rightarrow 0$ при $m\rightarrow\infty.$  Тоді в умовах і
позначеннях леми~\ref{lem1} можна обрати номер $k_0\in {\Bbb N},$
для якого існує $0<N=N(k_0, Q, D^{\,\prime})<\infty,$ незалежне від
$m$ і $f,$ таке що
$$M(\Gamma_m)\leqslant N,\qquad m\geqslant M_0=M_0(k_0)\,,$$
де $\Gamma_m$ -- сім'я кривих $\gamma:[0, 1]\rightarrow D$ в області
$D$ таких, що $f(\gamma)\in \Gamma(|\gamma_{1, m}|, |\gamma_{2, m}|,
D^{\,\prime}).$ }
\end{lemma}
\begin{proof} 
Нехай $k_0$ -- довільний номер, для якого виконується твердження
леми~\ref{lem1}. За означенням кривої $\gamma_{1, m}$ і сім'ї
$\Gamma_m$ ми можемо записати
\begin{equation}\label{eq7A}
\Gamma_m=\Gamma_m^1\cup \Gamma_m^2\,,
\end{equation}
де $\Gamma_m^1$ -- сім'я кривих $\gamma\in\Gamma_m$ таких, що
$f(\gamma)\in \Gamma(|\widetilde{\alpha}|, |\gamma_{2, m}|,
D^{\,\prime})$ і $\Gamma_m^2$ -- сім'я кривих $\gamma\in\Gamma_m$
таких, що $f(\gamma)\in \Gamma(|\widetilde{\alpha}_m|, |\gamma_{2,
m}|, D^{\,\prime}).$

\medskip
Враховуючи позначення леми~\ref{lem1}, покладемо
$$\varepsilon_0:=\min\{{\rm dist}\,(|\widetilde{\alpha}|, \overline{g_{k_0}}),
{\rm dist}\,(|\widetilde{\alpha}|, |\widetilde{\beta}|)\}>0\,.$$
Розглянемо тепер покриття множини $|\widetilde{\alpha}|$ наступного
вигляду: $\bigcup\limits_{x\in |\widetilde{\alpha}|}B(x,
\varepsilon_0/4).$ Оскільки $|\widetilde{\alpha}|$ є компактом в
$D^{\,\prime},$ знайдуться номери $i_1,\ldots, i_{N_0}$ такі, що
$|\widetilde{\alpha}|\subset \bigcup\limits_{i=1}^{N_0} B(z_i,
\varepsilon_0/4),$ де $z_i\in |\widetilde{\alpha}|$ при $1\leqslant
i\leqslant N_0.$ З огляду на~\cite[теорема~1.I.5.46]{Ku} легко
переконатися в тому, що
\begin{equation}\label{eq5A}
\Gamma(|\widetilde{\alpha}|, |\gamma_{2, m}|,
D^{\,\prime})>\bigcup\limits_{i=1}^{N_0} \Gamma(S(z_i,
\varepsilon_0/4), S(z_i, \varepsilon_0/2), A(z_i, \varepsilon_0/4,
\varepsilon_0/2))\,.
\end{equation}
Зафіксуємо $\gamma\in \Gamma_m^1,$ $\gamma:[0, 1]\rightarrow D,$
$\gamma(0)\in |\widetilde{\alpha}|,$ $\gamma(1)\in |\gamma_{2, m}|.$
Зі співвідношення~(\ref{eq5A}) випливає, що $f(\gamma)$ має підкриву
$f(\gamma)_1:=f(\gamma)|_{[p_1, p_2]}$ таку, що $$f(\gamma)_1\in
\Gamma(S(z_i, \varepsilon_0/4), S(z_i, \varepsilon_0/2), A(z_i,
\varepsilon_0/4,  \varepsilon_0/2))$$ при деякому $1\leqslant
i\leqslant N_0.$ Тоді $\gamma|_{[p_1, p_2]}$ є такою кривою, яка з
одного боку є підкривою $\gamma,$ а з іншого, належить до
сім'ї~$\Gamma_f(z_i, \varepsilon_0/4, \varepsilon_0/2),$ бо
$$f(\gamma|_{[p_1, p_2]})=f(\gamma)|_{[p_1, p_2]}\in\Gamma(S(z_i,
\varepsilon_0/4), S(z_i, \varepsilon_0/2), A(z_i, \varepsilon_0/4,
\varepsilon_0/2)).$$ Тим самим
\begin{equation}\label{eq6}
\Gamma_m^1>\bigcup\limits_{i=1}^{N_0} \Gamma_f(z_i, \varepsilon_0/4,
\varepsilon_0/2)\,.
\end{equation}
Покладемо
$$\eta(t)= \left\{
\begin{array}{rr}
4/\varepsilon_0, & t\in [\varepsilon_0/4, \varepsilon_0/2],\\
0,  &  t\not\in [\varepsilon_0/4, \varepsilon_0/2]
\end{array}
\right. \,.$$
Зауважимо, що функція~$\eta$ задовольняє
співвідношення~(\ref{eqA2}). Тоді, за визначенням відображення $f$
у~(\ref{eq2*A}), а також за співвідношенням~(\ref{eq6}) і з огляду
на напівадитивність модуля сімей кривих
(див.~\cite[теорема~6.2]{Va}), ми отримаємо, що
\begin{equation}\label{eq1}
M(\Gamma_m^1)\leqslant \sum\limits_{i=1}^{N_0} M(\Gamma_f(z_i,
\varepsilon_0/4, \varepsilon_0/2))\leqslant \sum\limits_{i=1}^{N_0}
\frac{N_04^n\Vert Q\Vert_1}{\varepsilon^n_0}\,,\qquad m\geqslant
M_0\,,
\end{equation}
where $\Vert Q\Vert_1=\int\limits_{D^{\,\prime}}Q(x)\,dm(x).$
Далі, по~\cite[теорема~1.I.5.46]{Ku} ми отримаємо, що
$$\Gamma_m^2>\Gamma_f(\overline{x_0}, r_{k_0+1}, r_{k_0})\,.$$
Міркуючи так, як і вище, покладемо
$$\eta(t)= \left\{
\begin{array}{rr}
1/(r_{k_0}-r_{k_0+1}), & t\in [r_{k_0+1}, r_{k_0}],\\
0,  &  t\not\in [r_{k_0+1}, r_{k_0}]
\end{array}
\right. \,.$$
Тоді з останнього співвідношення випливає, що
\begin{equation}\label{eq4D}
M(\Gamma_m^2)\leqslant \frac{\Vert
Q\Vert_1}{(r_{k_0}-r_{k_0+1})^n}\,, m\geqslant M_0\,.
\end{equation}
Отже, з~(\ref{eq7A}), (\ref{eq1}) і~(\ref{eq4D}), з огляду на
напівадитивність модуля сімей кривих, випливає, що
$$M(\Gamma_m)\leqslant
\left(\frac{N_04^n}{\varepsilon^n_0}+\frac{1}{(r_{k_0}-r_{k_0+1})^n}\right)\Vert
Q\Vert_1\,,\quad m\geqslant M_0\,.$$
Права частина останнього співвідношення не залежить від~$m,$ так що
ми можемо покласти
$N:=\left(\frac{N_04^n}{\varepsilon^n_0}+\frac{1}{(r_{k_0}-r_{k_0+1})^n}\right)\Vert
Q\Vert_1.$ Лему~\ref{lem4} повністю доведено.~$\Box$
\end{proof}

\medskip
{\bf 4. Доведення теореми~\ref{th2}.} Можливість неперервного
продовження відображення~$f\in {\frak S}_{\delta, A, Q }(D,
D^{\,\prime})$ на межу області $D$ є результатом теореми~\ref{th3}.
Одностайна неперервність сім'ї відображень~${\frak S}_{\delta, A, Q
}(D, D^{\,\prime})$ у внутрішніх точках області~$D$ є результатом
роботи~\cite[теорема~1.1]{SSD}.

\medskip
Покажемо одностайну неперервність сім'ї~${\frak S}_{\delta, A, Q
}(\overline{D}, \overline{D^{\,\prime}})$ на $\partial D.$
Припустимо протилежне. Тоді знайдуться точка $z_0\in \partial D,$
число~$\varepsilon_0>0,$ послідовність $z_m\in \overline{D}$ і
відображення $\overline{f}_m\in {\frak S}_{\delta, A, Q
}(\overline{D}, \overline{D^{\,\prime}})$ такі, що $z_m\rightarrow
z_0$ при $m\rightarrow\infty,$ при цьому,
\begin{equation}\label{eq12}
\rho(\overline{f}_m(z_m),
\overline{f}_m(z_0))\geqslant\varepsilon_0,\quad m=1,2,\ldots ,
\end{equation}
де $\rho$  -- одна з можливих метрик в~$\overline{D^{\,\prime}}_P,$
яку визначено за формулою типу~(\ref{eq1A}). Оскільки
$f_m=\overline{f}_m|_{D}$ продовжується по неперервності на
$\overline{D},$ ми можемо вважати, що $z_m\in D$ і, крім того,
знайдеться ще одна послідовність $z^{\,\prime}_m\in D,$
$z^{\,\prime}_m\rightarrow z_0$ при $m\rightarrow\infty,$ така що
$\rho(f_m(z^{\,\prime}_m), \overline{f}_m(z_0))\rightarrow 0$ при
$m\rightarrow\infty.$ В такому випадку, з~(\ref{eq12}) випливає, що
\begin{equation}\label{eq13}
\rho(f_m(z_m), f_m(z^{\,\prime}_m))\geqslant\varepsilon_0/2,\quad
m\geqslant m_0\,.
\end{equation}
Оскільки область~$D^{\,\prime}$ є регулярною, метричний
простір~$\overline{D^{\,\prime}}_P$ є компактним. Отже, можна
вважати, що послідовності~$f_m(z_m)$ і $f_m(z_m^{\,\prime})$
збігаються при $m\rightarrow\infty$ до деяких елементів~$P_1, P_2\in
\overline{D^{\,\prime}}_P,$ $P_1\ne P_2.$ Нехай $d_m$ і $g_m$ --
послідовності спадних областей, які відповідають простим кінцям
$P_1$ і $P_2,$ відповідно. З огляду на~\cite[лема~3.1]{IS}, див.
також~\cite[лема~1]{KR$_2$}, можна вважати, що послідовність
розрізів $\sigma_m,$ яка відповідає областям~$d_m,$ $m=1,2,\ldots, $
лежить на сферах~$S(\overline{x_0}, r_m),$ де $\overline{x_0}\in
\partial D^{\,\prime}$ і $r_m\rightarrow 0$ при $m\rightarrow\infty.$
Оберемо $x_0, y_0\in A$ так, що $x_0\ne y_0$ і $x_0\ne P_1\ne y_0,$
де континуум~$A\subset D^{\,\prime}$ взятий з умов
теореми~\ref{th2}. Без обмеження загальності, можна вважати, що
$d_1\cap g_1=\varnothing$ і $x_0, y_0\not\in d_1\cup g_1.$

\medskip
По лемам~\ref{lem1} і~\ref{lem4} знайдуться непересічні криві
$\gamma_{1,m}:[0, 1]\rightarrow D^{\,\prime}$ і $\gamma_{2,m}:[0,
1]\rightarrow D^{\,\prime},$ номер $M_0=M_0(k_0)>0$ і число $N> 0$
такі, що $\gamma_{1, m}(0)=x_0,$ $\gamma_{1, m}(1)=f_m(z_m),$
$\gamma_{2, m}(0)=y_0,$ $\gamma_{2, m}(0)=f_m(z^{\,\prime}_m),$
причому
\begin{equation}\label{eq15}
M(\Gamma_m)\leqslant N\,, m\geqslant M_0\,,
\end{equation}
де $\Gamma_m$ складається з тих і тільки тих кривих $\gamma$ в $D,$
для яких $f_m(\gamma)\in\Gamma(|\gamma_{1, m}|, |\gamma_{2, m}|,
D^{\,\prime})$ (див. малюнок~\ref{fig6}).
\begin{figure}
  \centering\includegraphics[width=350pt]{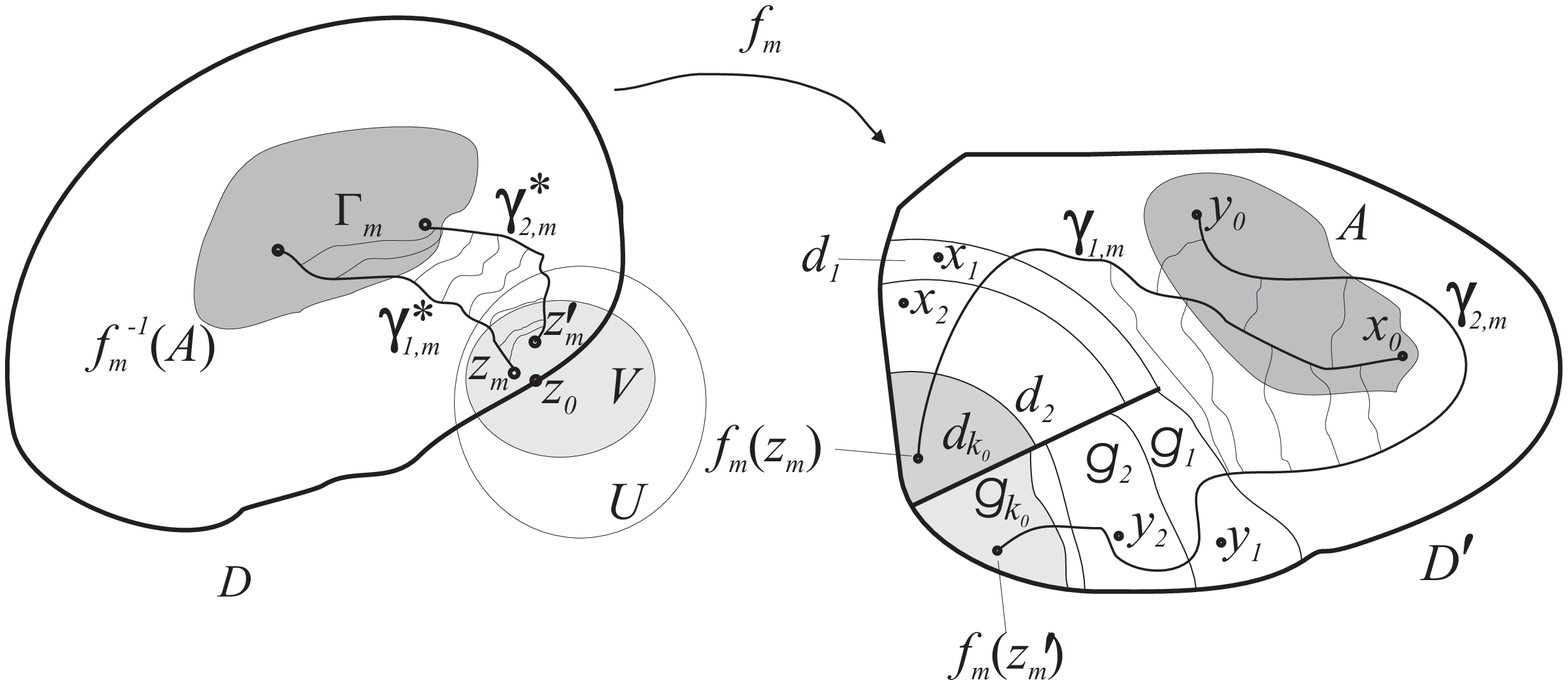}
  \caption{До доведення теореми~\ref{th2}.}\label{fig6}
 \end{figure}
З іншого боку, нехай~$\gamma^*_{1,m}$ і~$\gamma^*_{2,m}$ -- повні
підняття кривих~$\gamma_{1,m}$ і~$\gamma_{2,m}$ при відображенні
$f_m$ з початками в точках~$z_m$ і $z^{\,\prime}_m,$ відповідно
(такі підняття існують за~\cite[лема~3.7]{Vu$_1$}). Тоді
$\gamma^*_{1,m}(1)\in f^{\,-1}_m(A)$ і $\gamma^*_{2,m}(1)\in
f^{\,-1}_m(A)$ і, оскільки за умовою $h(f^{\,-1}_{m}(A), \partial
D)>\delta>0,$ $m=1,2,\ldots \,,$ ми будемо мати, що
$$h(|\gamma^*_{1, m}|)\geqslant h(z_m, \gamma^*_{1,m}(1)) \geqslant
(1/2)\cdot h(f^{\,-1}_m(A), \partial D)>\delta/2\,,$$
\begin{equation}\label{eq14}
h(|\gamma^*_{2, m}|)\geqslant h(z^{\,\prime}_m, \gamma^*_{2,m}(1))
\geqslant (1/2)\cdot h(f^{\,-1}_m(A), \partial D)>\delta/2
\end{equation}
для достатньо великих $m\in {\Bbb N}.$
Оберемо кулю $U:=B_h(z_0, r_0)=\{z\in\overline{{\Bbb R}^n}: h(z,
z_0)<r_0\},$ де $r_0>0$ і $r_0<\delta/4.$ Зауважимо, що
$|\gamma^*_{1, m}|\cap U\ne\varnothing\ne |\gamma^*_{1, m}|\cap
(D\setminus U)$ для достатньо великих $m\in{\Bbb N},$ оскільки
$h(f_m(|\gamma_{1, m}|))\geqslant \delta/2$ і $z_m\in|\gamma^*_{1,
m}|,$ $z_m\rightarrow z_0$ при $m\rightarrow\infty.$ Міркуючи
аналогічно, можна зробити висновок, що~$|\gamma^*_{2, m}|\cap
U\ne\varnothing\ne |\gamma^*_{2, m}|\cap (D\setminus U).$ Оскільки
$|\gamma^*_{1, m}|$ і $|\gamma^*_{2, m}|$ є континуумами, з огляду
на~\cite[теорема~1.I.5.46]{Ku}
\begin{equation}\label{eq8AA}
|\gamma^*_{1, m}|\cap \partial U\ne\varnothing, \quad |\gamma^*_{2,
m}|\cap
\partial U\ne\varnothing\,.
\end{equation}
Зафіксуємо $P:=N>0,$ де $N$ -- число зі співвідношення~(\ref{eq15}).
Оскільки межа області $D$ є слабо плоскою, знайдеться окіл $V\subset
U$ точки $z_0,$ такий що для будь-яких континуумів $E, F\subset D$ з
умовами $E\cap
\partial U\ne\varnothing\ne E\cap \partial V$ і $F\cap \partial
U\ne\varnothing\ne F\cap \partial V$ виконано нерівність
\begin{equation}\label{eq9AA}
M(\Gamma(E, F, D))>N\,.
\end{equation}
Зауважимо, що для достатньо великих $m\in {\Bbb N}$
\begin{equation}\label{eq10AA}
|\gamma^*_{1, m}|\cap \partial V\ne\varnothing, \quad |\gamma^*_{2,
m}|\cap
\partial V\ne\varnothing\,.\end{equation}
Дійсно, $z_m\in |\gamma^*_{1, m}|$ і $z^{\,\prime}_m\in
|\gamma^*_{2, m}|,$ де $z_m, z^{\,\prime}_m\rightarrow z_0\in V$ при
$m\rightarrow\infty.$ Отже, $|\gamma^*_{1, m}|\cap
V\ne\varnothing\ne |\gamma^*_{2, m}|\cap V$ для достатньо великих
$m\in {\Bbb N}.$ Крім того, $h(V)\leqslant h(U)=2r_0<\delta/2$ і,
оскільки по~(\ref{eq14}) $h(|\gamma^*_{1, m}|)>\delta/2,$ то
$|\gamma^*_{1, m}|\cap (D\setminus V)\ne\varnothing.$ Тоді
$|\gamma^*_{1, m}|\cap\partial V\ne\varnothing$
(див.~\cite[теорема~1.I.5.46]{Ku}). Аналогічно, $h(V)\leqslant
h(U)=2r_0<\delta/2$ і, оскільки по~(\ref{eq14}) $h(|\gamma^*_{2,
m}|)>\delta/2,$ то $|\gamma^*_{2, m}|\cap (D\setminus
V)\ne\varnothing.$ По~\cite[теорема~1.I.5.46]{Ku} ми отримаємо, що
$|\gamma^*_{1, m}|\cap\partial V\ne\varnothing.$ Отже,
(\ref{eq10AA}) встановлено. З огляду на~(\ref{eq9AA}), (\ref{eq8AA})
і (\ref{eq10AA}), ми отримаємо, що
\begin{equation}\label{eq6a}
M(\Gamma(|\gamma^*_{1, m}|, |\gamma^*_{2, m}|, D))>N\,.
\end{equation}
Нерівність~(\ref{eq6a}) суперечить~(\ref{eq15}), бо
$\Gamma(|\gamma^*_{1, m}|, |\gamma^*_{2, m}|, D)\subset \Gamma_m,$
отже,
$$M(\Gamma(|\gamma^*_{1, m}|, |\gamma^*_{2, m}|, D))
\leqslant M(\Gamma_m)\leqslant N\,.$$
Отримана суперечність вказує на невірність вихідного
припущення~(\ref{eq12}). Теорему доведено.~$\Box$

\medskip
{\bf 5. Деякі приклади.}

\begin{example}\label{ex1}
Отримаємо спочатку відображення, яке задовольняє умови і висновок
теореми~\ref{th3}. По-перше, розглянемо випадок, коли це
відображення є гомеоморфізмом, а функція $Q$ є обмеженою. Для
спрощення розглянемо плоский випадок. Нехай $D^{\,\prime}$ --
одиничний квадрат з викинутими відрізками $I_k=\{z=(x, y)\in {\Bbb
R}^2: x=1/k,\,\,0<y<1/2\},$ $k=2,3,\ldots,$ (див.
малюнок~\ref{fig7}).
\begin{figure}[h]
\centerline{\includegraphics[scale=0.5]{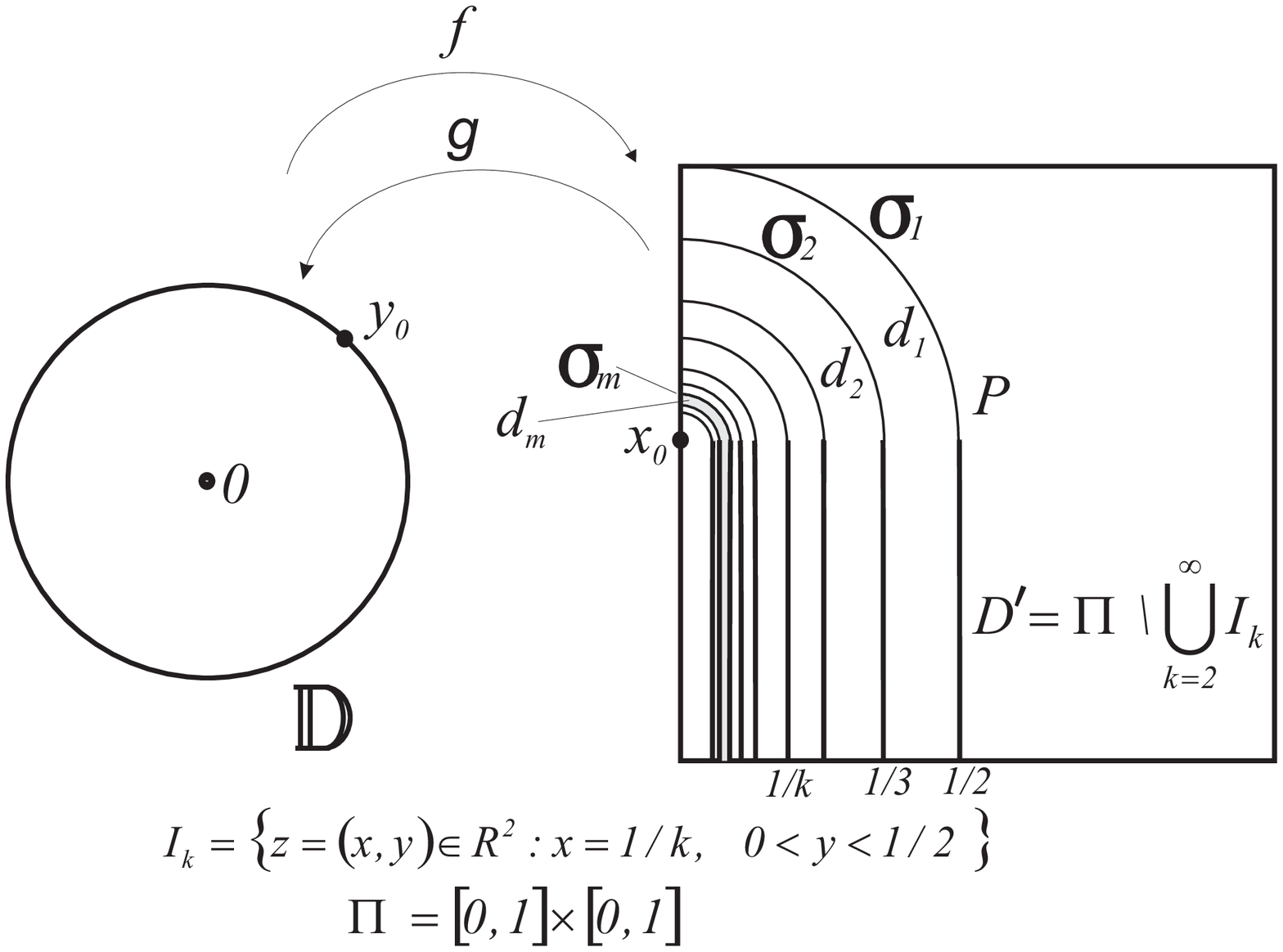}}
\caption{Ілюстрація до прикладу~\ref{ex1}}\label{fig7}
\end{figure}
Розглянемо простий кінець $P$ області $D^{\,\prime},$ створений за
допомогою розрізів
$$\sigma_m=\left\{z=x_0+\frac{e^{i\varphi}}{m+1},\,\, x_0=(0, 1/2),\,\, 0\leqslant
\varphi\leqslant \pi/2\right\}, \quad m=1,2,\ldots, .$$
Можна показати, що кінець $P$ дійсно є простим. За теоремою Рімана
про відображення, існує конформне відображення $g$ одиничного круга
${\Bbb D}=\{z\in {\Bbb C}: |z|<1\}$ на область $D^{\,\prime},$ крім
того, за теоремою Каратеодорі простому кінцю $P$ відповідає деяка
точка~$y_0\in\partial {\Bbb D}$ така, що $C(f, y_0)=I(P),$
$f=g^{\,-1},$ див.~\cite[теорема~9.4]{CL}. Отже, можна обрати
принаймні дві послідовності $z_k, w_k\in D^{\,\prime},$
$k=1,2,\ldots ,$ такі що $z_k, w_k\rightarrow P,$ $z_k\rightarrow
z_0$ і $w_k\rightarrow w_0$ при $k\rightarrow\infty,$ $z_0\ne w_0,$
причому $f(z_k)\rightarrow y_0$ і $f(w_k)\rightarrow y_0$ при
$k\rightarrow\infty.$ В цьому випадку, відображення $f:=g^{\,-1}$ не
має неперервного продовження в точку $y_0$ в поточковому сенсі, але
$g$ має неперервне продовження $\overline{g}:\overline{{\Bbb
D}}\rightarrow \overline{D^{\,\prime}}_P.$

\medskip
Оскільки $g$ -- конформне відображення, воно задовольняє
співвідношення~(\ref{eq2*A}) при $Q\equiv 1$ (див., напр.,
\cite[теорема~3.2]{MRV$_1$}).

\medskip
Зауважимо, що відображення $f$ задовольняє всі умови і висновок
теореми~\ref{th3}. Область ${\Bbb D}$ має слабо плоску межу (див.,
напр., \cite[теореми~17.10 і 17.12]{Va}), а область $D^{\,\prime}$ є
регулярною за означенням, крім того, функція $Q\equiv 1$ є
інтегровною в $D^{\,\prime}.$
\end{example}

\medskip
\begin{example}\label{ex4}
Для того, щоб тепер отримати аналогічне відображення з розгалуженням
у~(\ref{eq2*A}), покладемо
$$f_1(z)=(f\circ g)(z)\,,$$
де $g(z)=z^2.$ Зауважимо, що $K_O(z, f_1)=1$ і $N(f_1, {\Bbb
D})=2\,,$ тому $f_1$ також задовольняє співвідношення~(\ref{eq2*A})
з $Q\equiv 2.$ Знову таки, $f_1$ задовольняє всі умови
теореми~\ref{th3}.
\end{example}

\medskip
\begin{example}\label{ex3}
На основі прикладів~\ref{ex1} і~\ref{ex4} побудуємо тепер
відображення з розгалуженням, яке має необмежену характеристику, і
яке задовольняє всі умови і висновок теореми~\ref{th3}. Розглянемо
наступну конструкцію: нехай
$\varphi_1(z)=\frac{1}{e\sqrt{2}}(z-(1/2, 1/2)),$ $z\in
D^{\,\prime},$ тоді $\varphi$ переводить $D^{\,\prime}$ в деяку
область $D^{\,\prime\prime},$ що повністю лежить в крузі $B(0,
1/e).$ Цю область $D^{\,\prime\prime}$ перетворимо не деку іншу
однозв'язну область $D^{\,\prime\prime\prime}$ за допомогою
гомеоморфізму $\varphi_2(z)=\frac{z}{|z|\log\frac{1}{|z|}},$
$\varphi_2(0):=0.$ Тепер, цю область $D^{\,\prime\prime\prime}$
перетворимо за допомогою деякого конформного відображення
$\varphi_3$ на одиничний круг ${\Bbb D}$ (таке конформне
відображення існує завдяки теоремі Рімана). Нарешті, в ${\Bbb D}$
покладемо $\varphi_4(z)=z^2.$ Тепер розглянемо наступне відображення
\begin{equation}\label{eq1B}
F(z)=(\varphi^{\,-1}_1\circ\varphi^{\,-1}_2\circ\varphi^{\,-1}_3\circ\varphi_4)(z)\,.
\end{equation}
Тепер окремо розглянемо
$F_1(z)=(\varphi^{\,-1}_1\circ\varphi^{\,-1}_2)(z)$ і
$F_2(z)=(\varphi^{\,-1}_3\circ\varphi_4)(z).$ Передусім зауважимо,
що $K_O(F_1, z)=K_O(\varphi^{\,-1}_2, z),$ оскільки відображення
$\varphi^{\,-1}_1$ є конформним. Використовуючи техніку, застосовану
при розгляді~\cite[Proposition~6.3]{MRSY}, можна встановити, що
$\varphi^{\,-1}_2=\frac{z}{|z|}e^{-\frac{1}{|z|}},$ причому
$K_O(F_1, z)=K_O(\varphi^{\,-1}_2, z)=\frac{1}{|z|}.$ Тоді
$$K_O(F_1^{\,-1}(z), F_1)=K_O((\varphi_2\circ\varphi_1)(z), F_1)\,.$$
Маємо:
$$K_O((\varphi_2\circ\varphi_1)(z), F_1)=\frac{1}{|z|}\left|_{z\mapsto\frac{z-(1/2, 1/2)}{|z-(1/2, 1/2)|
\log\frac{e\sqrt{2}}{|z-(1/2, 1/2)|}}}\right.
=\log\frac{e\sqrt{2}}{|z-(1/2, 1/2)|}\,.$$
Зауважимо, що $F_1$ є відображенням класу $C^1$ в ${\Bbb
D}\setminus\{0\},$ крім того, якобіан $|J(z, f)|=|z|^3\cdot
e^{2/|z|}$ є локально обмеженим в ${\Bbb D}\setminus\{0\}.$ В такому
випадку, за наслідком~8.5 в~\cite{MRSY} відображення $F_1$ є
відображенням зі скінченним спотворенням довжини в ${\Bbb
D}\setminus\{0\}.$ Отже, за~\cite[теорема~8.5]{MRSY} відображення
$F_1$ задовольняє співвідношення
\begin{equation}\label{eq11A}
M(\Gamma^{\,\prime\prime})\leqslant
\int\limits_{D^{\,\prime}}Q(z)\cdot \rho_*(z)\,dm(z)
\end{equation}
для будь-якої сім'ї $\Gamma^{\,\prime\prime}$ локально спрямлюваних
кривих $\gamma$ в області $D^{\,\prime\prime}$ і будь-якої функції
$\rho_*\in {\rm adm}\,F_1(\Gamma^{\,\prime\prime}),$ де
$Q(z)=\log\frac{e\sqrt{2}}{|z-(1/2, 1/2)|}.$

\medskip
З іншого боку, відображення $F_2$ задовольняє співвідношення
\begin{equation}\label{eq2}
M(\Gamma)\leqslant 2\cdot M(F_2(\Gamma))\,,
\end{equation}
оскільки $N(F_2, {\Bbb D})=2$ і $K_O(F_2, z)=1$ (див.
\cite[теорема~3.2]{MRV$_1$}). Тоді об'єднуючи~(\ref{eq11A})
і~(\ref{eq2}), будемо мати
\begin{equation}\label{eq2A}
M(\Gamma)\leqslant \int\limits_{D^{\,\prime}}2Q(z)\cdot
\rho_*(z)\,dm(z)
\end{equation}
для будь-якої сім'ї $\Gamma$ локально спрямлюваних кривих $\gamma$ в
${\Bbb D}$ і будь-якої функції $\rho_*\in {\rm
adm}\,F_1(F_2(\Gamma))={\rm adm}\,F(\Gamma),$ де
$Q(z)=\log\frac{e\sqrt{2}}{|z-(1/2, 1/2)|}.$ Зауважимо, що функція
$Q(z)=\log\frac{e\sqrt{2}}{|z-(1/2, 1/2)|}$ є інтегровною в області
$D^{\,\prime}.$ Також зауважимо, що нерівність~(\ref{eq2A}) є
частковим випадком співвідношення~(\ref{eq2*A}), оскільки в
(\ref{eq2A}) сім'я кривих є будь-якою, отже, замість $\Gamma$ можна
взяти $\Gamma_f(y_0, r_1, r_2)$ як окремий випадок. Крім того, ми
можемо також покласти в~(\ref{eq2A}) $\rho_*(z)=\eta(|z-y_0|)$ при
$r_1<|z-y_0|<r_2$ і $z\in D^{\,\prime},$ $\rho_*(z)$ в інших
випадках. Якщо $\eta$ задовольняє~(\ref{eqA2}), то можна показати,
що $\rho_*(z)$ задовольняє~(\ref{eq2A}) для $\Gamma=\Gamma_f(y_0,
r_1, r_2)$ (див.~\cite[теорема~5.7]{Va}).

\medskip
Отже, всі умови теореми~\ref{th3} виконуються; відображення $F$
задовольняє всі умови цієї теореми і за цією теоремою продовжується
до відображення $\overline{F}:\overline{{\Bbb
D}}\rightarrow\overline{D^{\,\prime}}_P.$
\end{example}

\medskip
\begin{example}\label{ex2} Тепер побудуємо приклад, стосовний
теореми~\ref{th2}. Як відомо, дробово-лінійні автоморфізми
одиничного круга мають вигляд
$$f(z)=e^{i\theta}\frac{z-a}{1-\overline{a}z}, \quad z\in {\Bbb
D},\quad a\in{\Bbb D},\quad\theta\in [0, 2\pi)\,.$$
Покладемо $\theta=0,$ $a=1/n,$ $n=1,2,\ldots .$ В цьому випадку,
розглянемо сім'ю відображень
$\widetilde{f}_n(z)=\frac{z-1/n}{1-z/n}=\frac{nz-1}{n-z}.$ Нехай
$\widetilde{A}=[0, 1/2].$ Нехай $t\in [0, 1/2],$ тоді
$\widetilde{f}_n(t)=\frac{t-1/n}{1-t/n}.$ Оскільки похідна
$\widetilde{f}^{\,\prime}_n(t)=\frac{1-\frac{1}{n^2}}{(1-\frac
tn)^2}$ невід'ємна всюди, найменше значення функції
$\widetilde{f}_n(t)$ на $A$ буде точка $-1/n,$ найбільша --
$\frac{1/2-1/n}{1-1/2n}\rightarrow 1/2$ при $n\rightarrow\infty.$
Звідси існує $\delta>0$ таке, що $h(\widetilde{f}_n(\widetilde{A}),
\partial {\Bbb D})>\delta>0.$
Покладемо тепер $f_n:=\widetilde{f}^{\,-1}_n,$ і нехай
$A=F(\widetilde{A}),$ де $F$ -- відображення з прикладу~\ref{ex3}
(див. співвідношення~(\ref{eq1B})). Тоді сім'я відображень
$F_n:=F\circ f_n$ задовольняє всі умови і висновок
теореми~\ref{th2}. Зауважимо, що кожне з відображень $F_n$ не має
навіть неперервного евклідового продовження на одиничне коло, але
має це продовження як відображення $\overline{F_n}:\overline{{\Bbb
D}}\rightarrow \overline{D^{\,\prime}}_P.$ Більше того, сім'я
відображень $\{F_n\}_{n=1}^{\infty}$ є одностайно неперервною в
$\overline{D}.$
\end{example}


КОНТАКТНА ІНФОРМАЦІЯ

\medskip
\noindent{{\bf Євген Олександрович Севостьянов} \\
{\bf 1.} Житомирський державний університет ім.\ І.~Франко\\
кафедра математичного аналізу, вул. Велика Бердичівська, 40 \\
м.~Житомир, Україна, 10 008 \\
{\bf 2.} Інститут прикладної математики і механіки
НАН України, \\
вул.~Добровольського, 1 \\
м.~Слов'янськ, Україна, 84 100\\
e-mail: esevostyanov2009@gmail.com}

\end{document}